\documentclass[12pt]{amsart}

\usepackage{amssymb, amsfonts,amsmath}
\usepackage{amsthm}
\usepackage{xcolor}
\usepackage{enumerate}

\newtheorem{theorem}{Theorem}
\newtheorem{proposition}[theorem]{Proposition}
\newtheorem{lemma}[theorem]{Lemma}
\newtheorem{conjecture}[theorem]{Conjecture}

\newcommand{\rank}{\text{rank}}

\title[Dehn functions of lattices of type $A_n$ in products]{Dehn functions of higher rank arithmetic groups of type $A_n$ in products of simple Lie groups}
\author{Morgan Cesa}
\date{\today}

\begin{document}
\maketitle
\begin{abstract}
Suppose $\Gamma$ is an arithmetic group defined over a global field $K$, that the $K$-type of $\Gamma$ is $A_n$ with $n \geq 2$, and that the ambient semisimple group that contains $\Gamma$ as a lattice has at least two noncocompact factors.
We use results from Bestvina-Eskin-Wortman and Cornulier-Tessera to show that $\Gamma$ has a polynomially bounded Dehn function.
\end{abstract}
\section{Introduction}
Let $K$ be a global field, and $S$ a finite, nonempty set of inequivalent valuations on $K$. Denote by $\mathcal{O}_S$ the ring of $S$-integers in $K$, and let $K_v$ be the completion of $K$ with respect to $v \in S$. Let $\mathbf{G}$ be a noncommutative absolutely almost simple $K$-isotropic $K$-group, and let $G = \prod_{v\in S} \mathbf{G}(K_v)$. Note that $|S|$ is the number of simple factors of $G$, and that $\mathbf{G}(\mathcal{O}_S)$ is a lattice in $G$ under the diagonal embedding.

If $L$ is a field, the \emph{$L$-rank} of $\mathbf{G}$, denoted $\rank_L(\mathbf{G})$ is the dimension of a maximal torus in $\mathbf{G}(L)$. The \emph{geometric rank} of $G$ is $k(\mathbf{G}, S) = \sum_{v\in S} \rank_{K_v}(\mathbf{G})$. The Lie group $G$ is endowed with a left invariant metric, which we will denote $d_G$. Lubotzky-Mozes-Raghunathan showed that if $k(\mathbf{G}, S) \geq 2$, then the word metric on $\mathbf{G}(\mathcal{O}_S)$ is Lipschitz equivalent to the restriction of $d_G$ to $\mathbf{G}(\mathcal{O}_S)$ \cite{LubotzkyMozesRaghunathan2000}.

The following is a slight generalization of a conjecture due to Gromov \cite{Gromov1993}:
\begin{conjecture}\label{BEWconjecture}
If $k(\mathbf{G},S) \geq 3$, then the Dehn function of $\mathbf{G}(\mathcal{O}_S)$ is quadratic.
\end{conjecture}

Dru\c{t}u showed that if $k(\mathbf{G}, S) \geq 3$, $\rank_K(\mathbf{G}) = 1$, and $S$ contains only archimedean valuations, then the Dehn function of $\mathbf{G}(\mathcal{O}_S)$ is bounded above by the function $x \mapsto x^{2 + \epsilon}$ for any $\epsilon >0 $ \cite{Drutu1998}.

Young showed that if $\mathbf{G}(\mathcal{O}_S)$ is $\mathbf{SL_n}(\mathbb{Z})$ and $n \geq 5$ (i.e. $k(\mathbf{G},S) \geq~4$), then the Dehn function of $\mathbf{G}(\mathcal{O}_S)$ is quadratic \cite{Young2013}. Cohen showed that if $\mathbf{G}(\mathcal{O}_S)$ is $\mathbf{Sp_{2n}}(\mathbb{Z})$  and $n\geq 5$ (i.e. $k(\mathbf{G},S) \geq 5$), then the Dehn function of $\mathbf{G}(\mathcal{O}_S)$ is quadratic \cite{Cohen2014}. 
Bestvina-Eskin-Wortman showed that if $|S| \geq 3$ (that is, $G$ has at least 3 factors, which implies that $k(\mathbf{G},S) \geq 3)$), then the Dehn function of $\mathbf{G}(\mathcal{O}_S)$ is polynomially bounded \cite{BestvinaEskinWortman2013}.

In this paper, we will show:
\begin{theorem}\label{mainthm}
If the $K$-type of $\mathbf{G}$ is $A_n$, $ n \geq 2$, and $|S| \geq 2$, then the Dehn function of $\mathbf{G}(\mathcal{O}_S)$ is bounded by a polynomial of degree $3\cdot2^n$.
\end{theorem}
(Note that $n$ is the $K$-rank of $\mathbf{G}$, and therefore $k(\mathbf{G},S) \geq 4$.)

For example, Theorem \ref{mainthm} implies that the following groups have polynomially bounded Dehn functions: 
$\mathbf{SL_3}(\mathbb{Z}[\sqrt{2}])$, or more generally $\mathbf{SL_{n+1}}(\mathcal{O}_K)$ where $n \geq 2$, $\mathcal{O}_K$ is a ring of algebraic integers in a number field $K$, and $\mathcal{O}_K$ is not isomorphic to $\mathbb{Z}$ or $\mathbb{Z}[i]$;
$\mathbf{SL_{n+1}}(\mathbb{Z}[1/k])$ where $n \geq 2$ and $k \in \mathbb{N}-\{1\}$; and
$\mathbf{SL_{n+1}}(\mathbb{F}_p[t, t^{-1}])$ where $n \geq 2$ and $p$ is prime. Indeed, $\mathbf{SL_{n+1}}$ is of type $A_n$ regardless of the relative global field $K$, and $\mathbb{Z}[\sqrt{2}]$, $\mathcal{O}_K$, $\mathbb{Z}[1/k]$, and $\mathbb{F}_p[t, t^{-1}]$ are rings of $S$-integers with $|S|\geq 2$.

\subsection{Dehn Functions and Isoperimetric Inequalities}
If $H$ is a finitely presented group, and $w$ is a word in the generators of $H$ which represents the identity, then there is a finite sequence of relators which reduces $w$ to the trivial word. Let $\delta_H(w)$ be the minimum number of steps to reduce $w$ to the trivial word. The $\emph{Dehn function}$ of $H$ is defined as $$\delta_H(n) = \max_{length(w) \leq n} \delta_H(w)$$
While the Dehn function depends on the presentation of $H$, the growth class of the Dehn function is a quasi-isometry invariant of $H$.

The Dehn function of a simply connected $CW$-complex $X$ is defined analogously. For any loop $\gamma \subset X$, let $\delta_X(\gamma)$ be the minimal area of a disk in $X$ that fills $\gamma$. The Dehn function of $X$ is then
$$\delta_X(n) = \max_{length(\gamma) \leq n} \delta_X(\gamma)$$
If $X$ is quasi-isometric to $H$ (for example, if $X$ has a free, cellular, properly discontinuous, cocompact $H$-action), then the growth class of $\delta_X(n)$ is the same as that of $\delta_H(n)$.

\subsection{Coarse Manifolds}
An \emph{$r$-coarse manifold} in a metric space $X$ is the image of a map from the vertices of a triangulated manifold $M$ into $X$, with the property that any pair of adjacent vertices in $M$ are mapped to within distance $r$ of each other. We will abuse notation slightly and refer to the image of the map as an $r$-coarse manifold as well. If $\Sigma$ is a coarse manifold, then $\partial \Sigma$ is the restriction of the map to $\partial M$. If $M$ is an $n$-manifold, we will say $\Sigma$ is a coarse $n$-manifold, and we define the length or area of $\Sigma$ to be the number of vertices in $\Sigma$. We say two coarse $n$-manifolds, $\Sigma$ and $\Sigma'$, have the same topological type if the underlying manifolds $M$ and $M'$ have the same topological type.
\subsection{Bounds}
We will write $a = O(C)$ to mean that there is some constant $k$, which depends only on $G$ and $\mathbf{G}(\mathcal{O}_S)$, such that $a \leq kC$.

\subsection{Acknowledgements}
The author would like to thank her Ph.D. advisor, Kevin Wortman, under whose direction this work was carried out, for suggesting this problem and for his support and encouragement.
\section{Preliminaries}
\subsection{Parabolic Subgroups}\label{prelim}
Let $K$, $S$, and $\mathbf{G}$ be as above. There is a minimal $K$-parabolic subgroup $\mathbf{P} \leqslant \mathbf{G}$, and $\mathbf{P}$ contains a maximal $K$-split torus which we will call $\mathbf{A}$. Let $\Phi$ be the root system for $(\mathbf{G}, \mathbf{A})$, and observe that $\mathbf{P}$ determines a positive subset $\Phi^+ \subset\Phi$. Let $\Delta$ denote the set of simple roots in $\Phi^+$. (Note that $|\Delta| = rank_K(\mathbf{G})=n$.) For any $I\subseteq \Delta$, $[I]$ will denote the linear combinations generated by $I$ . Let $\Phi(I)^+ = \Phi^+ - [I]$ and $[I]^+=[I] \cap \Phi^+$. 
If $\alpha \in \Phi$, let $\mathbf{U}_{(\alpha)}$ be the corresponding root group. For any $\Psi \subseteq \Phi^+$ which is closed under addition, let $$\mathbf{U}_\Psi = \prod_{\alpha \in \Psi} \mathbf{U}_{(\alpha)}$$
Note that $$\prod_{ v\in S} \mathbf{U}_{\Psi}(K_{v})$$ can be topologically identified with a product of vector spaces and therefore can be endowed with a norm, $||\cdot||$. 

Let $\mathbf{A}_I$ be the connected component of the identity in $(\cap_{\alpha \in I} \ker(\alpha))$. The centralizer of $\mathbf{A}_I$ in $\mathbf{G}$, $\mathbf{Z}_\mathbf{G}(\mathbf{A}_I)$, can be written as $\mathbf{Z}_\mathbf{G}(\mathbf{A}_I)= ~\mathbf{M}_I\mathbf{A}_I$, where $\mathbf{M}_I$ is  a reductive $K$-group with $K$-anisotropic center. Notice that $\mathbf{M}_I\mathbf{A}_I$ normalizes $\mathbf{U}_{\Phi(I)^+}$, and $\mathbf{M}_I$ commutes with $\mathbf{A}_I$.
We define the standard parabolic subgroup $\mathbf{P}_I$ of $\mathbf{G}$ to be $$\mathbf{P}_I = \mathbf{U}_{\Phi(I)^+} \mathbf{M}_I \mathbf{A}_I$$ Note that $\mathbf{P}_{\emptyset} = \mathbf{P}$ and that when $\alpha \in \Delta$,  $\mathbf{P}_{\Delta-\alpha}$ is a maximal proper $K$-parabolic subgroup of $\mathbf{G}$.

We will use unbolding to denote taking the product over $S$ of the local points of a $K$-group, as in $$G = \prod_{v \in S} \mathbf{G}(K_v)$$ 

\subsection{Parabolic regions and reduction theory}
The following theorem was proved in different cases by Borel, Behr, and Harder (cf. \cite{Borel1991} Proposition 15.6, \cite{Behr1969} Satz 5 and Satz 8, and \cite{Harder1969} Korrolar 2.2.7). A summary of the individual results and how they imply the theorem is given in \cite{BestvinaEskinWortman2013}.
\begin{theorem}[Borel, Behr, Harder]\label{cosetreps}
There is a finite set $F\subseteq \mathbf{G}(K)$ of coset representatives for $\mathbf{G}(\mathcal{O}_S)\backslash \mathbf{G}(K) / \mathbf{P}(K)$.
\end{theorem}

Any proper $K$-parabolic subgroup $\mathbf{Q}$ of $\mathbf{G}$ is conjugate to $\mathbf{P}_I$ for some $I \subsetneq \Delta$. Let $$\Lambda_\mathbf{Q} = \{ \gamma f \in \mathbf{G}(\mathcal{O}_S)F | (\gamma f)^{-1} \mathbf{P}_I (\gamma f) = \mathbf{Q} \text{ for some } I \subsetneq \Delta\}$$
By Theorem \ref{cosetreps}, $\Lambda_{\mathbf{Q}}$ is nonempty. 
For $a\in A$ and $\alpha \in \Phi$, let $$|\alpha(a)| = \prod_{v \in S}|\alpha(a)|_v$$ where $|\cdot|_{v}$ is the norm on $K_{v}$. 
For $t>0$ and $I \subset \Delta$, let $$A_I^+(t) = \{ a\in A_I\mid |\alpha(a)|\geq t\text{ if } \alpha \in \Delta - I\}$$ and $A_I^+ = A_I^+(1)$. 
We define the\emph{ parabolic region} corresponding to $\mathbf{Q}$ to be $$R_{\mathbf{Q}}(t) = \Lambda_\mathbf{Q} U_{\Phi(I)^+} \mathbf{M}_I(\mathcal{O}_S) A_I^+(t)$$
The boundary of $A_I^+(t)$ is $$ \partial A_I^+(t) = \{ a \in A_I^+ \mid \text{ there exists }\alpha \in \Delta - I \text{ with } |\alpha(a)|\leq |\alpha(b)| \text{ for all } b \in A_I^+\}$$ and the boundary of the parabolic region $R_{\mathbf{Q}}(t)$ is $$\partial R_\mathbf{Q}(t) = \Lambda_\mathbf{Q} U_{\Phi(I)^+} \mathbf{M}_I(\mathcal{O}_S) \partial A_I^+(t)$$

For $0\leq m < |\Delta|$, let $\mathcal{P}(m)$ be the set of $K$-parabolic subgroups of $\mathbf{G}$ that are conjugate via $\mathbf{G}(K)$ to some $\mathbf{P}_I$ with $|I| = m$. The necessary reduction theory is proved in \cite{BestvinaEskinWortman2013}:
\begin{theorem}[Bestvina-Eskin-Wortman, 2013]\label{reduction}
There exists a bounded set $B_0 \subseteq G$, and given a bounded set $B_m \subseteq G$ and any $N_m \geq 0$ for $0 \leq m <|\Delta|$, there exists $t_m>1$ and a bounded set $B_{m+1}\subseteq G$ such that 
\begin{enumerate}[(i)]
\item $G = \cup_{\mathbf{Q} \in \mathcal{P}(0)} R_{\mathbf{Q}} B_0$;
\item if $\mathbf{Q}, \mathbf{Q}' \in \mathcal{P}(m)$ and $\mathbf{Q} \neq \mathbf{Q}'$, then the distance between $R_\mathbf{Q}(t_m)B_n$ and $R_{\mathbf{Q}'}(t_m)B_n$ is at least $N_m$;
\item \label{integers}$\mathbf{G}(\mathcal{O}_S) \cap R_\mathbf{Q}(t_m)B_m = \emptyset$ for all $m$;
\item\label{subtract} if $m \leq |\Delta| - 2$ then $(\cup_{\mathbf{Q}\in \mathcal{P}(m)} R_{\mathbf{Q}} B_{m}) - (\cup_{\mathbf{Q} \in \mathcal{P}(m)} R_\mathbf{Q}(2t_m)B_m)$ is contained in $\cup_{\mathbf{Q}\in\mathcal{P}(m+1)} R_\mathbf{Q} B_{m+1}$;
\item\label{final} $ (\cup_{\mathbf{Q}\in \mathcal{P}(|\Delta| - 1)} R_{\mathbf{Q}} B_{|\Delta|-1}) - (\cup_{\mathbf{Q} \in \mathcal{P}(|\Delta|-1)} R_{\mathbf{Q}}(2t_{|\Delta-1})B_{|\Delta|-1})$ 
is contained in $\mathbf{G}(\mathcal{O}_S)B_{|\Delta|}$; and
\item \label{quasiisometry} if $\mathbf{Q} \in \mathcal{P}(m)$, then there is an $(L, C)$ quasi-isometry $R_{\mathbf{Q}}(t_m) B_m \to U_{\Phi(I)^+} \mathbf{M}_I(\mathcal{O}_S) A_I^+$ for some $I \subsetneq \Delta$ with $|I| = m$. The quasi-isometry restricts to an $(L, C)$ quasi-isometry $\partial R_{\mathbf{Q}}(t_m)B_m \to U_{\Phi(I)^+}\mathbf{M}_I(\mathcal{O}_S)\partial A_I^+$ where $L \geq 1$ and $C \geq 0$ are independent of $\mathbf{Q}$. 
\end{enumerate}
\end{theorem}


\subsection{Filling coarse manifolds in the boundaries of parabolic regions}
For $I\subsetneq \Delta,$ we let $R_I = U_{\Phi(I)^+}\mathbf{M}_I(\mathcal{O}_S) A_I^+$.
\begin{proposition}\label{parabolicprop}
Suppose $I\subsetneq \Delta$ is a set of simple roots, and let $R_I$ denote the corresponding parabolic region of $\mathbf{G}$. Given $r>0$, there is some $r' \in \mathbb{R}^{>0}$ such that if $\Sigma \subset R_I$ is an $r$-coarse 2-manifold of area $L$, then there is an $r'$-coarse 2-manifold $\Sigma' \subset \partial R_I$ such that $\partial \Sigma = \partial \Sigma'$. Furthermore, if $|I|\leq |\Delta| - 2$, then $area(\Sigma') = O(L^2)$ and if $|I| = |\Delta| - 1$ then $area(\Sigma') = O(L^3)$.
\end{proposition}

Proposition \ref{parabolicprop} is proved in Sections \ref{secnonmax} (for nonmaximal parabolics) and \ref{secmax} (for maximal parabolics). 

That Proposition \ref{parabolicprop} implies Theorem \ref{mainthm} is essentially proved in Bestvina-Eskin-Wortman (see \cite{BestvinaEskinWortman2013} Sections 6 and 7). We restate it here in the specific case we require, and add explicit bounds on filling areas.
\begin{proof} [Proof of Theorem \ref{mainthm}]
Let $X$ be a simply connected $CW$-complex on with a free, cellular, properly discontinuous and cocompact $\mathbf{G}(\mathcal{O}_S)$-action. Let $x \in X$ be a basepoint, and define the orbit map
$$\phi:\mathbf{G}(\mathcal{O}_S) \to \mathbf{G}(\mathcal{O}_S)\cdot x$$
Note that $\phi$ is a bijective quasi-isometry between $\mathbf{G}(\mathcal{O}_S)$ with the left invariant metric $d_G$ and the orbit $\mathbf{G}(\mathcal{O}_S)\cdot x$ with the path metric from $X$.

Let $\ell\subset X$ be a cellular loop. The $\mathbf{G}(\mathcal{O}_S)$-action on $X$ is cocompact, so every point in $\ell$ is a uniformly bounded distance from $\mathbf{G}(\mathcal{O}_S)$. Therefore, there is a constant $r_0>0$ such that after a uniformly bounded perturbation, $\ell \cap \mathbf{G}(\mathcal{O}_S)$ is an $r_0$-coarse loop and the Hausdorff distance between $\ell$ and $\ell \cap \mathbf{G}(\mathcal{O}_S)$ is bounded. Let $L$ be the length of $\ell\cap\mathbf{G}(\mathcal{O}_S)$.

There is a constant $r_1>0$ which depends only on $r_0$ and the quasi-isometry constants of $\phi$ such that $\phi^{-1}(\ell\cap\mathbf{G}(\mathcal{O}_S))$ is an $r_1$-coarse loop in $\mathbf{G}(\mathcal{O}_S)$. Since $G$ is quasi-isometric to a $CAT(0)$ space (a product of Euclidean buildings and symmetric spaces), there is an $r_1$-coarse disk $D \subset G$ with $\partial D = \phi^{-1}(\ell\cap\mathbf{G}(\mathcal{O}_S)\cdot x)$ and area $O(L^2)$.

Set $D = D_0$ and $N_0 = 2r_1$. Let $B_1$ and $t_0$ be as in Theorem \ref{reduction}. If $\mathbf{Q} \in \mathcal{P}(0)$, let $$D_{0, \mathbf{Q}} = D_0 \cap R_\mathbf{Q}(t_0)B_0$$
Note that $D_{0,\mathbf{Q}}$ and $D_{0,\mathbf{Q'}}$ are disjoint if $\mathbf{Q} \neq \mathbf{Q'}$. For each $\mathbf{Q} \in \mathcal{P}(0)$, we can perturb $D_{0, \mathbf{Q}}$ by at most $r_1$ to ensure that $\partial D_{0, \mathbf{Q}} \subset \partial R_\mathbf{Q}(t_0)B_0$. By Proposition \ref{reduction}(\ref{quasiisometry}), $\partial R_{\mathbf{Q}}(t_0)B_0$ is quasi-isometric to $\partial R_{\emptyset}$. By Proposition \ref{parabolicprop}, there is some $r_2$ depending only on $r_1$ and the quasi-isometry constants and an $r_2$-coarse 2-manifold $D'_{0, \mathbf{Q}} \subset \partial R_\mathbf{Q}(t_0)B_0$ such that the $2$-manifold obtained by replacing each $D_{0, \mathbf{Q}}$ by $D'_{0, \mathbf{Q}}$, 
$$D_1 = \left(D_0 - \bigcup_{\mathbf{Q}\in\mathcal{P}(0)} D_{0, \mathbf{Q}}\right) \bigcup \left(\bigcup_{\mathbf{Q}\in\mathcal{P}(0)} D'_{0, \mathbf{Q}}\right)$$
is an $r_1$-coarse 2-disk, and $area(D_1) = O(area(D_0)^2) = O(L^4)$. 

By Proposition \ref{reduction}(\ref{subtract}), $$D_1 \subset\left( \bigcup_{\mathbf{Q}\in\mathcal{P}(0)} R_\mathbf{Q}B_0\right)  -\left( \bigcup_{\mathbf{Q}\in\mathcal{P}(0)} R_{\mathbf{Q}}(2t_0)B_0 \right)\subset \bigcup_{\mathbf{Q} \in \mathcal{P}(1)} R_{\mathbf{Q}}B_1$$

By Proposition \ref{reduction}(\ref{integers}), $\mathbf{G}(\mathcal{O}_S) \cap R_{\mathbf{Q}}(t_0)B_0 = \emptyset$, and therefore $\partial D_0 = \partial D_1$.

For $1\leq m \leq |\Delta| - 1$ repeat the above process with $m$ in place of 0, to obtain an $r_{m+1}$-coarse disk $D_{m+1}$ with $\partial D_{m+1} = \partial D_0$ and $area(D_{m+1}) = O(L^{k_{m+1}})$, where $k_{m+1}= 2^{m+2}$ if $n \leq |\Delta| - 2$ and $k_{|\Delta|} = 3\cdot 2^{|\Delta|}$. Furthermore, $$D_{m + 1} \subset \bigcup_{\mathbf{Q} \in \mathcal{P}(m)} R_{\mathbf{Q}} B_m - \bigcup_{\mathbf{Q}\in\mathcal{P}(m)} R_\mathbf{Q}(2t_m) B_m$$
which implies that $D_{|\Delta|} \subset \mathbf{G}(\mathcal{O}_S) B_{|\Delta|}$ by Proposition \ref{reduction}(\ref{final}).

Since $\mathbf{G}(\mathcal{O}_S) B_{|\Delta|}$ is finite Hausdorff distance from $\mathbf{G}(\mathcal{O}_S)$, there is some $r' > 0$ such that there is an $r'$-coarse disk $D'\subset \mathbf{G}(\mathcal{O}_S)$ with $\partial D' = \phi^{-1}(\ell\cap\mathbf{G}(\mathcal{O}_S)\cdot x)$ and $area(D') = O(L^k)$, where $k = 3\cdot 2^{|\Delta|}$.

There is some $r''>0$ which depends only on $r'$ and the quasi-isometry constants of $\phi$ such that $\phi(D') \subset X$ is an $r''$-coarse disk with boundary $\ell \cap \mathbf{G}(\mathcal{O}_S)\cdot x$. First connect pairs of adjacent vertices in $\phi(D')$ by 1-cells to obtain $D''$, then add 2-cells whose 1-skeleton is in $D''$ to obtain $D'''$. Note that $\partial D''' = \ell$, $D'''$ is a bounded Hausdorff distance to $\phi(D')$, and the number of cells in $D'''$ is $O(area(D')) = O(L^k)$ where $k = 3 \cdot 2^{|\Delta|}$. Recall that $|\Delta| = \rank_K(\mathbf{G})=n$, so the Dehn function of $\mathbf{G}(\mathcal{O}_S)$ is bounded by a polynomial of degree $3 \cdot 2^{n}$.

\end{proof}
\subsection{More preliminaries}

\begin{lemma} \label{hausdist} Given $r>0$ sufficiently large, $I \subseteq \Delta$, and $S' \subsetneq S$, there is some $a \in \mathbf{A}_I(\mathcal{O}_S)$ that strictly contracts all root subgroups of $\prod_{v \in S'} \mathbf{U}_{\Phi(I)^+}(K_v)$, such that $d_G(a, 1) \leq r$.
\end{lemma}

\begin{proof}
Lemma 12 in Bestvina-Eskin-Wortman \cite{BestvinaEskinWortman2013} shows that the projection of $\mathbf{A}_I(\mathcal{O}_S)$ to $\prod_{v \in S'} \mathbf{A}_I(K_v)$ is a finite Hausdorff distance from $\prod_{v \in S'} \mathbf{A}_I(K_v)$ . (The proof is independent of $|S|$.) This implies that there is some $a \in \mathbf{A}_I(\mathcal{O}_S)$ such that $|\alpha(a)|_v < 1$ for all $\alpha \in \Delta - I$ and $v \in S'$. Therefore, if $u \in \prod_{v \in S'} \mathbf{U}_{(\beta)}(K_v)$ for some $\beta \in \Phi(I)^+$, then  $||a^{-1}ua|| < ||u||$.
\end{proof}

We will make use of the following lemma in both the maximal and nonmaximal parabolic cases:
\begin{lemma}\label{unipotentpaths}
Let $r>0$ be sufficiently large and $I \subset \Delta$. If $u \in U_{\Phi(I)^+}$, then there is an $r$-coarse path $p_u \subset U_{\Phi(I)^+} \mathbf{A}^+_I(\mathcal{O}_S)$ joining $u$ to $1$ such that $length(p_u) = O(d_G(u, 1))$. 
\end{lemma}
\begin{proof}
Let $L = d_G(u, 1)$, and notice that $||u|| \leq O(e^L)$. Letting $S = \{v_1,\dots, v_k\}$, we can write $u = (u_1,\dots,u_k)$, where $u_i \in \mathbf{U}_{\Phi(I)^+}(K_{v_i})$. 

By the bound on $||u||$, we also have $||u_i||\leq O(e^L)$. By Lemma \ref{hausdist}, we can choose $a_i \in \mathbf{A}_I^+(\mathcal{O}_S)$ such that $a_i$ strictly contracts $\mathbf{U}_{\Phi(I)^+}(K_{v_i})$ and $d_G(a_i, 1) \leq r$.

For some $T_i = O(L)$, $d_G(a_i^{-T_i} u_i a_i^{T_i}, 1) = d_G( u_i a_i^{T_i}, a_i^{T_i}) \leq 1$. Let $p_i = \{ a_i^k \mid 0\leq k \leq T_i\} \cup \{ua_i^k\mid 0\leq k \leq T_i\}$. Note that $p_i$ is an $r$-coarse path from $1$ to $u_i$ of length $O(L)$. Taking $$p_u =p_1 \cup \left( \bigcup_{2 \leq i \leq k} (u_1, \ldots, u_{i - 1}, 1, \ldots, 1)\cdot p_i\right)$$
 gives the desired path from $1$ to $u$.
\end{proof}

\section{Nonmaximal Parabolic Subgroups}\label{secnonmax}
In this section, we will prove Proposition \ref{parabolicprop} in the case where $|I|~\leq~|\Delta|~-~2$.

First, we will divide $\partial R_I$ into two pieces. Recall that $$\partial R_I = U_{\Phi(I)^+} \mathbf{M}_I(\mathcal{O}_S) \partial A_I^+$$ 
$$\partial A_I^+ = \{ a \in A_I^+ | \text{ there exists } \alpha \in \Delta - I\text{ with } |\alpha(a)| \leq |\alpha(b)| \text{ for all } b \in A_I^+\}$$
For $\alpha \in \Delta - I$, we define $A_{I, \alpha}^+, Z_{I, \alpha}^+, B_{I, \alpha},$ and $\hat{B}_{I,\alpha}$ as follows: $$A_{I, \alpha}^+ =\left \{ a \in A_I^+ \mid|\alpha(a)| \leq |\alpha(b)| \text{ for all } b \in A_I^+\right \}$$
$$Z_{I, \alpha}^+ = \bigcup _{\beta \in \Delta -( I\cup \alpha)} A_{I, \beta}^+$$
$$B_{I, \alpha} = U_{\Phi(I)^+} \mathbf{M}_I(\mathcal{O}_S) A_{I, \alpha}^+$$ 
$${\hat B}_{I, \alpha} = U_{\Phi(I)^+}\mathbf{M}_I(\mathcal{O}_S) Z_{I, \alpha}^+$$
Note that $\partial A_I^+ = \cup_{ \alpha \in \Delta - I} A_{I, \alpha}^+$ and that $\partial R_I = B_{I, \alpha} \cup \hat{B}_{I, \alpha}$. We also observe that $A_{I, \alpha}^+ \neq A_{I\cup\alpha}^+$, since $\mathbf{A}_I(\mathcal{O}_S) \subseteq A_{I, \alpha}^+$ for any $\alpha \in \Delta - I$, but $\mathbf{A}_I(\mathcal{O}_S) \not \subset A_{I\cup\alpha}^+$ in general.

Since $A_I^+$ is quasi-isometric to a Euclidean space, there is a projection to $\partial A_I^+$ which is distance nonincreasing. Note that $\mathbf{M}_I(\mathcal{O}_S)$ commutes with $A_I^+$, so there is a distance nonincreasing map $\mathbf{M}_I(\mathcal{O}_S) A_I^+ \to \mathbf{M}_I(\mathcal{O}_S) \partial A_I^+$. 
Let $\pi_I: R_I \to \partial R_I$ be the composition of the distance nonincreasing maps $U_{\Phi(I)^+} \mathbf{M}_I(\mathcal{O}_S) A_I \to \mathbf{M}_I(\mathcal{O}_S)A_I^+$ and $\mathbf{M}_I(\mathcal{O}_S)A_I^+ \to \mathbf{M}_I(\mathcal{O}_S) \partial A_I^+$.

\begin{lemma} \label{division} Suppose $I \subsetneq \Delta$ is a set of simple roots such that $|I|\leq |\Delta| - 2$ and let $r > 0$ and $\alpha \in \Delta - I$ be given. If $\Sigma \subset R_I$ is an $r$-coarse $2$-manifold with boundary and $\partial \Sigma \subset \partial R_I$, then  $\Sigma=\Sigma_1 \cup \Sigma_2$ for $r$-coarse $2$-manifolds with boundary, $\Sigma_1$ and $\Sigma_2,$ such that 
\begin{enumerate}[(i)]
\item $\pi_I(\partial\Sigma_1) \subset B_{I, \alpha}$ and $\pi_I(\partial\Sigma_2) \subset \hat{B}_{I, \alpha}$, 
\item $\Sigma_1 \cap \partial \Sigma \subset B_{I, \alpha}$ and $\Sigma_2 \cap \partial \Sigma\subset \hat{B}_{I, \alpha}$, and
\item $\Sigma_1 \cap \Sigma_2$ consists of finitely many $r$-coarse paths $p_1, \dots, p_k$, with $\pi_I(p_i) \subset\partial B_{I, \alpha} $ and finitely many $r$-coarse loops $\gamma_1, \dots, \gamma_n$ with $\pi_I(\gamma_l) \subset \partial B_{I,\alpha}$.
\end{enumerate}
\end{lemma}

\begin{proof}
By transversality, $\pi_I(\Sigma)$ intersects $\partial B_{I,\alpha}$ in an $r$-coarse 1-manifold which is made up of finitely many $r$-coarse paths $(\bar{p}_1, \dots, \bar{p}_k)$ with endpoints in $\pi_I(\partial\Sigma)$ and finitely many $r$-coarse loops $(\bar{\gamma}_1,\dots,\bar{\gamma}_n)$ which do not intersect $\pi_I(\partial\Sigma)$. 
Furthermore, $\pi_I(\Sigma)$ intersects $B_{I,\alpha}$ (respectively $\hat{B}_{I,\alpha}$) in a 2-manifold with boundary, $\bar{\Sigma}_1$ (respectively $\bar{\Sigma}_2$), and
\begin{align}\partial\bar{\Sigma}_i &= (\bar{\Sigma}_i \cap \pi_I(\partial\Sigma)) \cup(\bar{p}_1\cup \dots\cup \bar{p}_k) \cup (\bar{\gamma}_1\cup\dots\cup\bar{\gamma}_n)\end{align}
 
For $x \in \partial R_I$, note that $\pi_I(x) \in B_{I,\alpha}$ if and only if $x \in B_{I,\alpha}$ (since $\pi_I$ only changes the unipotent coordinates of points in $\partial R_I$). Let $\Sigma_1$ and $\Sigma_2$ be the respective preimages of $\bar{\Sigma}_1$ and $\bar{\Sigma}_2$ under $\pi_I$ restricted to $\Sigma$. Note that $\bar{p}_i$ and $\bar{\gamma}_i$ lift to $r$-coarse paths and loops $p_i$ and $\gamma_i$ in $\Sigma$.
Conclusion (i) holds because $\bar{\Sigma}_i= \pi_I(\Sigma_i)$, and conclusions (ii) and (iii) hold by (1) and the definition of $p_i$ and $\gamma_l$.
\end{proof}

\begin{lemma}\label{homotopy}
Suppose $I \subsetneq \Delta$ is a set of simple roots such that $|I|\leq |\Delta| - 2$ and let $r > 0$ and $\alpha \in \Delta - I$ be given. 
If $\Omega$ is a closed $r$-coarse $1$-manifold in $B_{I,\alpha}$ or $\hat{B}_{I,\alpha}$ with diameter and distance to $B_{I,\alpha}\cap\hat{B}_{I,\alpha}$ bounded by $L$, then there is an $r'$-coarse $2$-manifold $\mathcal{A} \subset \partial R_I$ such that $\partial \mathcal{A} = \Omega \cup u\pi_I(\Omega)$ for some $u \in U_{\Phi(I)^+}$ and $area(\mathcal{A}) = O(L^2)$.
\end{lemma}

\begin{proof}
We will begin with the case where $\Omega \subset B_{I,\alpha}$.
For $x \in \Omega$, we can write $x = u_x m_x a_x$ with $u_x \in U_{\Phi(I)^+}, m_x \in \mathbf{M}_I(\mathcal{O}_S),$ and $a_x \in A_{I, \alpha}^+$. Since the diameter of $\Omega$ is bounded by $L$, $||u_x^{-1}u_y|| \leq O(e^L)$ for any $x, y \in \Omega$. 
Choose $b \in int(A_{I\cup\alpha}^+)$ with $d_G(b, 1) \leq r$. Note that $b$ commutes with $U_{[I\cup\alpha]}$, $\mathbf{M}_I(\mathcal{O}_S)$, and $A_{I}^+$, and that conjugation by $b^{-1}$ strictly contracts $U_{\Phi(I\cup\alpha)^+}$. Also, $U_{\Phi(I)^+} = U_{\Phi(I\cup\alpha)^+} U_{[I\cup\alpha]\cap \Phi(I)^+}$, so conjugation by $b^{-1}$ does not expand any root group in $U_{\Phi(I)^+}$.

Since $d_G(b, 1) \leq r$, left invariance of $d_G$ implies that $d_G(gb, g) \leq r$ for any $g \in G$. Right multiplication by $b^k$ is distance nonincreasing on $\Omega$ when $k \geq 0$, since for any $x, y \in \Omega$, 
\begin{align*}
d_G(xb^k, yb^k) &= d_G(u_xm_xa_x b^k,u_ym_ya_yb^k)\\
&= d_G(u_x b^k m_x a_x, u_y b^k m_y a_y)\\
&= d_G(b^{-k}u_y^{-1}u_xb^k m_x a_x, m_ya_y)\\
&\leq d_G(u_y^{-1}u_xm_xa_x, m_ya_y)\\
&= d_G(u_x m_x a_x, u_y m_y a_y) \\
&= d_G(x, y)
\end{align*}
Therefore, $\cup _{0 \leq k \leq m} \Omega b^k$ is a $2r$-coarse $2$-manifold for any $m \in \mathbb{N}$, which has the topological type of $\Omega\times [0,1]$, boundary $\Omega \cup \Omega b^n$ and whose area is bounded by $Lm$. There is some $T = O(L)$ such that the $U_{\Phi(I\cup\alpha)^+}\text{-coordinates}$ of $\Omega b^T$ are nearly constant. More precisely, there is some fixed $u^* \in U_{\Phi(I\cup\alpha)^+}$ and some $v_x \in U_{[I\cup\alpha]\cap\Phi(I)^+}$ for each $x$ such that $$d_G(u_xm_xa_xb^T, u^*v_xm_xa_xb^T) \leq r$$ for every $x$ in $\Omega$. Let $\Omega_1 = \{u^*v_xm_xa_xb^T\}_{x \in \Omega}$ and let $\mathcal{A}_1 = \Omega_1 \cup ~(\cup_{0 \leq k \leq T} \Omega b^k)$ be the $2r$-coarse $2$-manifold with boundary $\Omega \cup \Omega_1$. Note that $area(\mathcal{A}_1) = O(L^2)$.

Let $\Omega_2 = \{ u^* v_x m_x a_x\}_{x \in \Omega}$. Note that $\Omega_2$ is an $r$-coarse 1-manifold of the same diameter as $\Omega$, since \begin{align*}d_G(u^* v_x m_x a_x, u^* v_y m_y a_y) &= d_G(u^*b^T v_x m_x a_x, u^* b^T v_y m_y a_y) \\
&= d_G( u^* v_x m_x a_x b^T, u^* v_y m_y a_y b^T)  \\
&\leq r\end{align*}
Again, there is a $2r$-coarse 2-manifold formed by $\cup_{0\leq k \leq T} \Omega_1 b^{-k}$, with area $O(L^2)$ and boundary $\Omega_1 \cup \Omega_2$.
After left translation, $(u^*)^{-1}\Omega_2 \subset U_{[I\cup\alpha]\cap\Phi(I)^+} \mathbf{M}_I(\mathcal{O}_S) A_{I,\alpha}^+$. 
Since $b$ commutes with 
$U_{[I\cup\alpha] \cap\Phi(I)^+} \mathbf{M}_I(\mathcal{O}_S) A_{I,\alpha}^+$, after a perturbation by at most $r$, the $2r$-coarse 2-manifold formed by $\cup_{k \in \mathbb{Z}} (u^*)^{-1}\Omega_2b^k$ intersects $\partial B_{I, \alpha}$ in a $2r$-coarse closed $1$-manifold of length $O(L)$. 
Call this $\Omega_3$ and let $\mathcal{A}_3$ be the portion of $\cup_{k \in \mathbb{Z}} (u^*)^{-1}\Omega_2b^k$ bounded by $(u^*)^{-1}\Omega_2$ and $\Omega_3$. 
Since the distance from $\Omega$ to $\partial B_{I,\alpha}$ is bounded by $L$, the area of $\mathcal{A}_3$ is $O(L^2)$. Note that if $\hat{x}=v_xm_xa_x \in ~(u^*)^{-1}\Omega_2$, then $\bar{x}=v_xm_x\bar{a}_x \in \Omega_3$, where $\bar{a}_x \in \partial A^+_{I,\alpha}$. The bound on the diameter of $\Omega_3$ implies that $||v_x^{-1}v_y||\leq e^L$ for all $\bar{x} \in \Omega_3$. 

Choose $c \in \partial A_I^+$ such that $d_G(c, 1) \leq r$, and for every $v \in S$, $|\alpha(c)|_v >1 $ and $|\beta(c)|_v = 1$ for every $\beta \in \Delta - \alpha$. There is some $T'=O(L)$ such that $\Omega_3c^{T'}$ has nearly constant $U_{[I\cup\alpha]\cap\Phi(I)^+}$-coordinates. That is, there is some $v^* \in U_{[I\cup\alpha]\cap\Phi(I)^+}$ such that $d_G(v_xm_x\bar{a}_x c^{T'}, v^* m_x \bar{a}_x c^{T'}) \leq 2r$ for all $\bar{x} \in \Omega_3$. 
Let $\Omega_4 = \{ v^* m_x \bar{a}_x c^{T'}\}_{ x\in\Omega}$, and let $\mathcal{A}_4$ be the $4r$-coarse 2-manifold $\Omega_4 \cup (\cup_{0\leq k \leq T'} \Omega_3c^k)$. The area of $\mathcal{A}_4$ is $O(L^2)$.
Since $c$ commutes with $\mathbf{M}_I(\mathcal{O}_S)$ and $A_{I}^+$, $\Omega_5 = \Omega_4 c^{-T'}$ is a $2r$-coarse $1$-manifold, and there is a $4r$-coarse 2-manifold $\mathcal{A}_5 = \cup_{0\leq k \leq T'} \Omega_4 c^{-k}$ which has boundary $\Omega_4 \cup \Omega_5$, and area $O(L^2)$.

Finally, observe that $\Omega_5 = \{ v^* m_x \bar{a}_x\}_{x \in \Omega}$ has the same $\mathbf{M}_I(\mathcal{O}_S)$-coordinates as $\Omega$, and that $b$ commutes with $\Omega_5$. Therefore, there is a $2r$-coarse 1-manifold $\Omega_6 \subset \cup_{k\in\mathbb{Z}}\Omega_5 b^k $ which has the form $\Omega_6 = \{ v^* m_x a_x\}_{x\in\Omega}$, and there is a $4r$-coarse 2-manifold $\mathcal{A}_6$ bounded by $\Omega_5$ and $\Omega_6$ with area $O(L^2)$.

Taking $$\mathcal{A} = \mathcal{A}_1 \cup \mathcal{A}_2 \cup u^*\mathcal{A}_3 \cup u^* \mathcal{A}_4 \cup u^* \mathcal{A}_5 \cup u^* \mathcal{A}_6$$ and $r' = 4r$ completes the proof.
\end{proof}

\begin{lemma} \label{intersectionpath}
Suppose $I \subsetneq \Delta$ is a set of simple roots such that $|I|\leq |\Delta| - 2$, and let $\alpha \in \Delta - I$ and $r>0$ be given. If $p \subset R_I$ is an $r$-coarse path with endpoints $x, y \in \partial B_{I,\alpha}$ such that $\pi_I(p) \subset \partial B_{I,\alpha}$, then there is an $r$-coarse path $p' \subset \partial B_{I,\alpha}$ joining $x$ to $y$ of length $O(length(p))$, and $\pi_I(p) \cup \pi_I(p')$ bound a disk of area $O(length(p)^2)$ in $\partial R_I$.
\end{lemma}

\begin{proof}
Let $length(p) = L$. We can write $x = u_xm_xa_x$ and $y=u_ym_ya_y$ for $u_x, u_y \in U_{\Phi(I)^+}; m_x, m_y \in \mathbf{M}_I(\mathcal{O}_S);$ and $a_x, a_y \in \partial A_{I,\alpha}^+$.

Since $\pi_I$ is distance nonincreasing, $\pi_I(p)$ is an $r$-coarse path of length $L$ from $m_xa_x$ to $m_y a_y$. Left multiplication by $u_x$ gives an $r$-coarse path $p_1$, with length $L$, joining $x$ to $u_x m_ya_y$. 

Note that $u'= (m_ya_y)^{-1}(u_x^{-1}u_y)(m_ya_y) \in U_{\Phi(I)^+}$ because $U_{\Phi(I)^+}$ is normalized by $\mathbf{M}_I(\mathcal{O}_S) A_I^+$. Also,
\begin{align*}
d_G(u', 1) &= d_G(u_xm_ya_y, u_ym_ya_y)\\
&\leq d_G(u_xm_ya_y, u_xm_xa_x) + d_G(u_xm_xa_x, u_ym_ya_y)\\
&\leq 2L
\end{align*}
By Lemma \ref{unipotentpaths}, there is an $r$-coarse path in $ U_{\Phi(I)^+}\mathbf{A}_I^+(\mathcal{O}_S)$ from $u'$ to 1, with length $O(L)$. Left multiplication by $u_xm_ya_y$ gives an $r$-coarse path $p_2 \subset m_ya_yU_{\Phi(I)^+}\mathbf{A}_I^+(\mathcal{O}_S)$ of length $O(L)$ joining $u_xm_ya_y$ to $y$. 

Let $p' = p_1 \cup p_2$. Note that $p \cup p'$ is a loop in $R_I$, and that $\pi_I(p_1) = \pi_I(p)$. Therefore $\pi_I(p_2)$ forms a loop in $m_y\mathbf{A}_I^+(\mathcal{O}_S)$. Since $\mathbf{A}_I^+(\mathcal{O}_S)$ is quasi-isometric to a Euclidean space of dimension $(|\Delta-I|)(|S|-1)$, it has a quadratic Dehn function, and therefore $\pi_I(p_2)$ bounds an $r$-coarse disk of area $O(L^2)$ in $m_y\mathbf{A}_I^+(\mathcal{O}_S)\subset \partial R_I$.
\end{proof}

We will now prove Proposition \ref{parabolicprop} in the case when $|I|\leq |\Delta| - 2$.
\begin{proof}[Proof of Proposition \ref{parabolicprop} for nonmaximal parabolics]
We will prove the lemma in two cases: first the case where $\partial \Sigma$ intersects both $B_{I, \alpha}$ and $\hat{B}_{I, \alpha}$ nontrivially for some $\alpha\in\Delta-I$; second the case where $\partial \Sigma \subset B_{I, \alpha}$ for some $\alpha\in \Delta-I$. These two cases are sufficient, because $\partial R_I = \cup_{\alpha \in \Delta - I}B_{I,\alpha}$, so $\partial \Sigma$ must intersect $B_{I, \alpha}$ for at least one $\alpha\in\Delta-I$. 

Suppose $\Sigma$ intersects both $B_{I,\alpha}$ and $\hat{B}_{I,\alpha}$. By Lemma \ref{division}, $\Sigma$ can be written as the union of two $r$-coarse 2-manifolds, $\Sigma_1$ and $\Sigma_2$, such that $\Sigma_1 \cap \partial \Sigma \subset B_{I,\alpha}$ and $\Sigma_2 \cap \partial \Sigma \subset \hat{B}_{I,\alpha}$, and $\Sigma_1 \cap \Sigma_2$ is a collection of $r$-coarse loops and $r$-coarse paths in $R_I$ with endpoints in $\partial \Sigma$. 

Suppose $p_j$ is an $r$-coarse path in $\Sigma_1 \cap \Sigma_2$, with endpoints in $\partial B_{I,\alpha}$. Lemma \ref{division} implies that $\pi_I(p_j) \subset \partial B_{I,\alpha}$, so we can apply Lemma \ref{intersectionpath} to obtain an $r$-coarse path $p'_j$ in $\partial B_{I,\alpha}$ which has the same endpoints as $p_j$ and length $O(length(p_j))$. If $\gamma_l$ is an $r$-coarse loop in $\Sigma_1 \cap \Sigma_2$, choose $x_l \in \gamma_l$ and write $x_l = u_lg_l$ for $u_l \in U_{\Phi(I)^+}$ and $g_l \in \mathbf{M}_I(\mathcal{O}_S)A_I^+$. Let $\gamma'_l = u_l\pi_I(\gamma_l)$ and note that $\gamma'_l \subset \partial B_{I,\alpha}$ and $\pi_I(\gamma'_l) = \pi_I(\gamma_l)$.

Note that $\partial \Sigma_i$ is a closed 1-manifold, and $$\partial \Sigma_i = (\Sigma_i \cap \partial \Sigma) \cup (p_1 \cup \dots \cup p_k)\cup(\gamma_1\cup\dots\cup\gamma_n)$$
Although $\partial \Sigma_i \not \subset \partial R_I$, we can replace $p_j$ by $p'_j$ and $\gamma_l$ by $\gamma'_l$ to obtain a closed $1$-manifold of the same topological type as $\partial \Sigma_i$ which is contained in $\partial R_I$. Let 
$$\Omega_i = (\Sigma_i \cap \partial \Sigma) \cup (p'_1\cup \dots \cup p'_k)\cup(\gamma'_1\cup\dots\cup\gamma'_n)$$
By Lemmas \ref{division} and \ref{intersectionpath}, the total length of $\Omega_i$ is $O(area(\Sigma))$.

Lemma \ref{homotopy} implies the existence of a constant $r'>0$ and $r'$-coarse $2$-manifolds $\mathcal{A}_1$ and $\mathcal{A}_2$ such that $\partial \mathcal{A}_i = \Omega_i \cup u_i \pi_I(\Omega_i)$ for some $u_i \subset U_{\Phi(I)^+}$, and $area(\mathcal{A}_i) = O(area(\Sigma)^2)$. By Lemma \ref{intersectionpath}, there is a family of disks $D_{i,j} \subset \partial R_I$ such that $$\Sigma_i' = \mathcal{A}_i \cup \left(\cup_{j}D_{i,j}\right) \cup u_i \pi_I(\Sigma_i)$$ is an $r'$-coarse 2-manifold of the same topological type as $\Sigma_i$.  Note that $\sum_{j=1}^k length(p_j) \leq L$, which implies that $\sum_{j=1}^k area(D_{i, j}) \leq L^2$ and therefore $area(\Sigma_i') = O(area(\Sigma)^2)$. Taking $\Sigma' = \Sigma'_1 \cup \Sigma'_2$ completes the first case of the proof.

We now assume that $\partial \Sigma \subset B_{I,\alpha}$. Let $\Omega = \partial \Sigma$ and let $L$ be the total length of $\partial \Sigma$. Every point $x \in \partial \Sigma$ can be written as $x = u_xm_xa_x$ for $u_x \in U_{\Phi(I)^+}$, $m_x \in \mathbf{M}_I(\mathcal{O}_S)$, and $a_x \in A_{I,\alpha}^+$. Note that $||u_x^{-1}u_y|| = O(e^L)$ for $x, y \in \partial \Sigma$.
Choose some $b \in int(A_{I\cup\alpha}^+)$ which strictly contracts $U_{\Phi(I\cup\alpha)^+}$.
As in the proof of Lemma \ref{homotopy}, right multiplication by $b^k$ is distance nonincreasing on $\Sigma$ when $k \geq 0$, and there is some $T = O(L)$ such that $\Omega b^T$ has nearly constant $U_{\Phi(I\cup\alpha)^+}$-coordinates. Let $u^* \in U_{\Phi(I\cup\alpha)^+}$ be such that $$d_G(u_xm_xa_xb^T, u^*v_xm_xa_xb^T) \leq r$$ for every $x \in \Omega$. Let $\Omega_1 = \{u^* v_xm_xa_x|x \in \Omega\}$. As in the proof of Lemma \ref{homotopy}, there is a $2r$-coarse 2-manifold $\mathcal{A}$ with boundary $\Omega \cup \Omega_1$ and area $O(L^2)$. 

There is a distance nonincreasing map $f:U_{\Phi(I)^+}\mathbf{M}_I(\mathcal{O}_S)A_I^+ \to U_{[I\cup\alpha]\cap\Phi(I)^+} \mathbf{M}_I(\mathcal{O}_S) A_{I,\alpha}^+$. Taking $r' = 2r$ and $\Sigma' = f(\Sigma) \cup \mathcal{A}$ completes the proof.
\end{proof}

\section{Maximal Parabolic Subgroups}\label{secmax}

In this section, we will prove Proposition \ref{parabolicprop} in the case where $R_I$ is a maximal parabolic subgroup of $G$ (when $|I|=|\Delta|-1$). There is a simple root $\alpha \in \Delta$ such that $I = \Delta - \alpha$.

As in the previous section, there is a distance nonincreasing map $\pi_I: U_{\Phi(I)^+} \mathbf{M}_I(\mathcal{O}_S) A_I \to \mathbf{M}_I(\mathcal{O}_S) \partial A_I$. Note that $\partial A_I = A_{\Delta}$ which is quasi-isometric to $\mathbf{A}(\mathcal{O}_S)$, so $\mathbf{M}_I(\mathcal{O}_S) \partial A_I$ is quasi-isometric to $(\mathbf{M}_I\mathbf{A})(\mathcal{O}_S)$. 

\begin{lemma} \label{pathstom} Given $r>0$ sufficiently large, and  $x \in \partial R_{I}$, with $d_G(x, 1)$ bounded by $L$, there is an $r$-coarse path in $\partial R_{I}$ joining $x$ to $\pi_I(x)$ which has length $O(L)$. 
\end{lemma}
\begin{proof}
We can write $x = uma$ for $u \in U_{\Phi(I)^+}$, $m \in\mathbf{M}_I(\mathcal{O}_S)$ and $a \in \mathbf{A}(\mathcal{O}_S)$. Then $\pi_I(x) = ma$. Note that $(\mathbf{M}_I \mathbf{A})(\mathcal{O}_S)$ normalizes $U_{\Phi(I)^+}$. So finding an $r$-coarse path from $x$ to $\pi_I(x)$ of length $O(L)$ can be reduced to the problem of finding an $r$-coarse path from $(ma)^{-1} u (ma) \in U_{\Phi(I)^+}$ to $1$ of length $O(L)$. Since $||(ma)^{-1}u(ma)|| \leq O(L)$, Lemma \ref{unipotentpaths} completes the proof.
\end{proof}

Fix some $w \in S$. Let $\mathbf{T}_I$ be a $K$-defined $K$-anisotropic torus in $\mathbf{M}_I$ such that $g\mathbf{T}_Ig^{-1}=\mathbf{M}_I\cap \mathbf{A}$. Since $\mathbf{T}_I$ is $K$-anisotropic, Dirichlet's unit's theorem tells us that $\mathbf{T}_I(\mathcal{O}_S)$ is cocompact in $T_I$, so in particular, the projection of $\mathbf{T}_I(\mathcal{O}_S)$ to $\mathbf{T}_I(K_w)$ is a finite Hausdorff distance from $\mathbf{T}_I(K_w)$. Let $\widehat{T}_I$ be the projection of $\mathbf{T}_I(\mathcal{O}_S)$ to $\mathbf{T}_I(K_w)$.

\begin{lemma} \label{toruscontraction} Suppose $\beta \in \Phi(I)^+$, so that $\mathbf{U}_{(\beta)}(K_{w}) \leqslant \mathbf{U}_{\Phi(I)^+}(K_{w})$. There is some $t \in \widehat{T}_I$ such that $gtg^{-1}$ strictly contracts $\mathbf{U}_{(\beta)}(K_{w})$.
\end{lemma}

\begin{proof}
It suffices to show that there is some $t' \in \mathbf{M}_I(K_{w}) \cap \mathbf{A}(K_w)$ which strictly contracts $\mathbf{U}_{(\beta)}(K_w)$. 

We first note that since the $K$-type of $\mathbf{G}$ is $A_n$, $\Delta = \{ \alpha_1,\dots, \alpha_n\}$, and a general root $\gamma \in \Phi$ has the form $$\gamma = \pm \sum_{i = j}^k \alpha_i$$ where $1 \leq j \leq k \leq n$. 
Because $P_I$ is a maximal parabolic, $I = \Delta - \alpha_m$ for some $m$ such that $1 \leq m \leq n$.

Let $\Delta_1 = \{\alpha_1, \dots, \alpha_{m-1}\}$ and $\Delta_2 = \{\alpha_{m+1}, \dots, \alpha_n\}$. At least one of these sets must be nonempty. We will assume that $\Delta_2$ is non-empty for the sake of simplicity. We can write $\mathbf{M}_I = \mathbf{M}_1 \times \mathbf{M}_2$, where $$\mathbf{M}_1 = \langle \mathbf{U}_{(\alpha_i)}, \mathbf{U}_{(-\alpha_i)}\rangle _{i <m}$$
$$\mathbf{M}_2 = \langle \mathbf{U}_{(\alpha_i)}, \mathbf{U}_{(-\alpha_i)}\rangle _{i >m}$$

Let $\mathbf{A}_i = \mathbf{A} \cap \mathbf{M}_i$, and note that $\mathbf{P}_\emptyset \cap \mathbf{M}_i$ is a minimal parabolic subgroup of $\mathbf{M}_i$, $\mathbf{A}_i$ is a maximal $K$-split torus in $\mathbf{P}_\emptyset\cap\mathbf{M}_i$, and $\Delta_i$ is the set of simple roots with respect to $\mathbf{A}_i$.

Since $\beta \in \Phi(\Delta-\alpha_m)^+$, we know that $$\beta = \alpha_{j} + \dots + \alpha_m + \dots +\alpha_{k}$$ for fixed choices of $j$ and $k$ such that $1 \leq j \leq m \leq k \leq n$. 

Suppose that $k > m$, and choose $a \in \mathbf{A}^+_2(K_w)$ such that $|\alpha_i(a)|_w <1$ for all $\alpha_i \in \Delta_2$. Note that $|\alpha_i(a)|_w = 1$ for $\alpha_i \in \Delta_1$, since $a \in \mathbf{M}_2(K_w)$.

Conjugation by $a$ acts on $\mathbf{U}_{(\beta)}(K_w)$ by scalar multiplication by the constant $$C= \prod_{i=j}^k|\alpha_i(a)|_w$$
By our choice of $a$, we know that $C= |\alpha_m(a)|_wC'$ where $C' < 1$. If $|\alpha_m(a)|_w < \frac{1}{C'}$, then $C<1$, and $a$ contracts $\mathbf{U}_{(\beta)}(K_w)$ by a factor of $C$. If $|\alpha_m(a)|_w > \frac{1}{C'}$, then $C>1$ and $a^{-1}$ contracts $\mathbf{U}_{(\beta)}(K_w)$ by a factor of $\frac{1}{C}$. (Note that either $a$ or $a^{-1}$ must contract $U_{(\gamma)}(K_w)$ for any other $\gamma \in \Phi(I)^+$ with $k > m$.)

If $C = 1$, choose $a' \in \cap_{i = 1}^m \ker(\alpha_i)$ such that $|\alpha_i(a')|_w \leq 1$ for all $\alpha_i \in \Delta_2$ and $|\alpha_k(a')|_w < 1$. Note that $$\prod_{i=j}^k |\alpha_i(aa')|_w = C \prod_{i=m+1}^k |\alpha_i(a')|_w < C$$ so $aa'$ contracts $\mathbf{U}_{(\beta)}(K_w)$.

If $\beta = \alpha_j + \dots + \alpha_m$, a different approach is required. Consider the group $$\mathbf{M}_3 = \langle \mathbf{U}_{\alpha_m}, \mathbf{U}_{-\alpha_m}, \mathbf{U}_{\alpha_{m+1}}, \mathbf{U}_{-\alpha_{m+1}}\rangle$$
and let $\mathbf{A}_3 = \mathbf{M}_3 \cap \mathbf{A}$. Note that $\Delta_3 = \{ \alpha_m, \alpha_{m+1}\}$ is the set of simple roots of $\mathbf{M}_3$, and the $K$-type of $\mathbf{M}_3$ is $A_2$. Furthermore, $\alpha_m$ determines a maximal parabolic subgroup $\mathbf{P}^* \leqslant\mathbf{M}_3$, with $\ker(\alpha_m) =\mathbf{P}^*\cap \mathbf{A}_3$.

Let $L = \langle \mathbf{U}_{\alpha_{m+1}}(K_w), \mathbf{U}_{-\alpha_{m+1}}(K_w)\rangle$, and choose $a \in L\cap \mathbf{A}_3(K_w)$ with $|\alpha_{m+1}(a)|_w < 1$. We argue that $a$ contracts $\mathbf{U}_{(\beta)}(K_w)$. Since $L \cap \mathbf{A}_1(K_w)$ is trivial, $|\alpha_i(a)|_w = 1$ for all $i < m$. So the action of $a$ on $\mathbf{U}_{(\beta)}(K_w)$ depends only on $|\alpha_m(a)|_w$. Let $\phi$ be the $K$-automorphism of $\mathbf{M}_3$ which stabilizes $\mathbf{A}_3$ and transposes $\mathbf{P}^*$ with its opposite with respect to $\mathbf{A}_3$. Note that $\ker(\alpha_m) \cap L$ is trivial, since $\phi$ preserves $L$ but does not preserve $\mathbf{P}^*$.
 Therefore, $|\alpha_m(a)|_w \neq 1$, and after possibly replacing $a$ by its inverse, we find that $a$ contracts $\mathbf{U}_{(\beta)}(K_w)$ by a factor of $|\alpha_m(a)|_w$.

\end{proof}

\begin{lemma} \label{fillinsubgroup}
The Dehn function of $U_{\Phi(I)^+} \widehat{T}_I\mathbf{A}_I(\mathcal{O}_S)$ is quadratic.
\end{lemma}
\begin{proof}
We observe that $\widehat{T}_I\mathbf{A}_I(\mathcal{O}_S)$ is a free abelian group. Also, $U_{\Phi(I)^+}$ is normalized by $\widehat{T}_I\mathbf{A}_I(\mathcal{O}_S)$, and since the $K$-type of $\mathbf{G}$ is $A_n$, $U_{\Phi(I)^+}$ is abelian and $\mathbf{U}_{\Phi(I)^+}(K_v)$ isomorphic to a direct sum of one or more copies of $K_{v}$. 

Therefore, $U_{\Phi(I)^+}\widehat{T}_I\mathbf{A}_I(\mathcal{O}_S)$ can be written as $$\bigoplus_{v\in S} \mathbf{U}_{\Phi(I)^+}(K_v) \rtimes \widehat{T}_I\mathbf{A}_I(\mathcal{O}_S)$$
By Theorem 3.1 in \cite{CornulierTessera2010}, it suffices to show that for any two unipotent coordinate subgroups, $\mathbf{U}_{(\beta_1)}(K_{v})$ and $\mathbf{U}_{(\beta_2)}(K_{v'})$, of $U_{\Phi(I)^+}$, there is some element of  $\widehat{T} \mathbf{A}_I(\mathcal{O}_S)$ which simultaneously contracts $\mathbf{U}_{(\beta_1)}(K_{v})$ and $\mathbf{U}_{(\beta_2)}(K_{v'})$.

If $v = v'$, then $\mathbf{U}_{(\beta_1)}(K_{v})$ and $\mathbf{U}_{(\beta_2)}(K_{v'})$ are contained in the same factor of $U_{\Phi(I)^+}$. By Lemma \ref{hausdist}, there is some $a \in \mathbf{A}_I(\mathbf{O}_S)$ which simultaneously contracts $\mathbf{U}_{(\beta_1)}(K_{v})$ and $\mathbf{U}_{(\beta_2)}(K_{v'})$.

If $v \neq v'$, then $\mathbf{U}_{(\beta_1)}(K_{v})$ and $\mathbf{U}_{(\beta_2)}(K_{v'})$ are in different factors of $U_{\Phi(I)^+}$. In this case, either $|S| \geq 3$ or $|S| = 2$. If $|S|\geq 3$, then we may again apply Lemma \ref{hausdist} to obtain $a \in \mathbf{A}_I(\mathcal{O}_S)$ which simultaneously contracts $\mathbf{U}_{\Phi(I)^+}(K_v) \times \mathbf{U}_{\Phi(I)^+}(K_{v'})$. 

If $|S| = 2$, we may assume that $v = w$.
Let $g \in \mathbf{M}_I(K_{w})\times \{1\}$ be the element which diagonalizes $\widehat{T}_I$. Note that $g$ commutes with $\mathbf{A}_I(\mathcal{O}_S)$ and normalizes $U_{\Phi(I)^+}$, so $U_{\Phi(I)^+} \widehat{T}_I\mathbf{A}_I(\mathcal{O}_S)$ is conjugate to $U_{\Phi(I)^+} (g\widehat{T}_Ig^{-1})\mathbf{A}_I(\mathcal{O}_S)$, and it suffices to prove the lemma for the latter group.

By Lemma \ref{toruscontraction}, there is some $gtg^{-1} \in g\widehat{T}_Ig^{-1}$ which contracts $\mathbf{U}_{(\beta_1)}(K_{w})$ and commutes with $\mathbf{U}_{(\beta_2)}(K_{v'})$.
There is some $a \in \mathbf{A}_I(\mathcal{O}_S)$ which contracts $\mathbf{U}_{(\beta_2)}(K_{v'})$. If $a$ expands $\mathbf{U}_{(\beta_1)}(K_{w})$, then there is a positive power of $gtg^{-1}$ such that $gt^kg^{-1}a$ simultaneously contracts $\mathbf{U}_{(\beta_1)}(K_{w})$ and $\mathbf{U}_{(\beta_2)}(K_{v'})$.
\end{proof}

\begin{proof}[Proof of Proposition \ref{parabolicprop} for maximal parabolics]

Since $\pi_I$ is distance nonincreasing, $\pi_I(\Sigma)$ is a $2$-manifold in $\partial R_I$ with area $O(L^2)$, so if we can create an annulus between $\partial \Sigma$ and $\pi_I(\partial\Sigma)$ which has area $O(L^3)$, then taking $\Sigma'$ to be the union of this annulus with $\pi_I(\Sigma)$ completes the proof.
By Lemma \ref{pathstom}, there is a path from each point in $\partial\Sigma$ to its image in $\pi_I(\partial\Sigma)$ which has length $O(L)$. Two adjacent points in $\partial\Sigma$, along with their images in $\pi_I(\partial\Sigma)$ and these two paths give a loop of length $O(L)$ in $U_{\Phi(I)^+} \mathbf{A}_I(\mathcal{O}_S) B$ where $B$ is a ball in $\mathbf{M}_I(\mathcal{O}_S)$ of radius $r$ around $1$. 
Note that this subset of $G$ is quasi-isometric to $U_{\Phi(I)^+} \mathbf{A}_I(\mathcal{O}_S)$, and by Lemma \ref{fillinsubgroup}, these loops have quadratic fillings in $\partial R_I$. 
Since there are $O(L)$ such loops formed by adjacent pairs of points in $\partial\Sigma$, this gives an annulus $\mathcal{A}$ with $\partial \mathcal{A}= \partial\Sigma \cup \pi_I(\partial\Sigma)$, and $area(\mathcal{A}) = O(L^3)$, completing the proof.
\end{proof}

\bibliography{general}
\bibliographystyle{alpha}
\end{document}